\documentclass[11pt]{article}

\usepackage[table]{xcolor}
\usepackage{amsmath,amsthm,amsfonts,amssymb,amscd, amsxtra,mathrsfs,commath,url,easyReview}
\usepackage{hyperref}
\usepackage[doc]{optional}
\usepackage{enumitem}
\usepackage{pstricks}
\usepackage{mathtools,esvect}
\usepackage{cite}
\usepackage{color}
\usepackage{subcaption}
\usepackage{soul}
\definecolor{labelkey}{rgb}{0,0.08,0.45}
\definecolor{refkey}{rgb}{0,0.6,0.0}

\usepackage{float}
\usepackage{threeparttablex,color,soul,colortbl,hhline,rotating,booktabs,longtable,multirow,makecell}
\definecolor{lightgray}{gray}{0.95}

\newcommand{\myscale}{0.45}

\usepackage[left]{lineno}
\usepackage{blindtext}

\usepackage{fancyhdr}

\usepackage[margin=1.0in]{geometry}

\newcommand{\ttrial}{t_{\rm trial}}
\newcommand{\tnew}{t_{\rm new}}

\newcommand{\fenv}[1]%
{\ensuremath{\,\overrightarrow{\operatorname{env}}_{#1}}}
\newcommand{\benv}[1]%
{\ensuremath{\,\overleftarrow{\operatorname{env}}_{#1}}}

\newcommand{\RR}{\ensuremath{\mathbb R}}

\newcommand{\NN}{\ensuremath{\mathbb N}}

\newcommand{\dom}{\ensuremath{\operatorname{dom}}}
\newcommand{\im}{\textrm{Im}}

\newcommand{\argmax}{\ensuremath{\operatorname*{argmax}}}

{\begin{list}{}{%
\settowidth{\labelwidth}{\textrm{#1~}}%
\setlength{\leftmargin}{\labelwidth+\labelsep}}}
{\end{list}}
\newtheorem{theorem}{Theorem}[section]
\newtheorem{lemma}[theorem]{Lemma}

\newtheorem{definition}[theorem]{Definition}
\newtheorem{remark}[theorem]{Remark}

\def\doi{DOI}

\newcounter{count}

\allowdisplaybreaks

\newcommand{\la}{{\langle}}
\newcommand{\ra}{{\rangle}}

\begin{document}

\title{\textrm{A Proximal Gradient Method with an Explicit Line search for Multiobjective Optimization}}
\author{
Y.\ Bello-Cruz\thanks{Department of Mathematical Sciences, Northern Illinois University. Watson Hall 366, DeKalb, IL, USA - 60115. E-mail:
\texttt{yunierbello@niu.edu.}}\and J. G. Melo\thanks{Institute of Mathematics and Statistics, Federal University of Goias, Campus II- Caixa
    Postal 131, CEP 74001-970, Goi\^ania-GO, Brazil.  E-mail:
\texttt{jefferson@ufg.br}} \and L. F. Prudente\thanks{Institute of Mathematics and Statistics, Federal University of Goias, Campus II- Caixa
    Postal 131, CEP 74001-970, Goi\^ania-GO, Brazil.  E-mail:
\texttt{lfprudente@ufg.br}} \and R. V. G. Serra\thanks{Department of Mathematics, Federal University of Piau\'i, CEP 64049550, Teresina-PI, Brazil. E-mail:
\texttt{rayserra@ufpi.edu.br}}}

\maketitle

\begin{abstract} 
 \noindent 
 We present a proximal gradient method for solving convex multiobjective optimization problems, where each objective function is the sum of two convex functions, with one assumed to be continuously differentiable. The algorithm incorporates a backtracking line search procedure that requires solving only one proximal subproblem per iteration, and is exclusively applied to the differentiable part of the objective functions. Under mild assumptions, we show that the sequence generated by the method convergences to a weakly Pareto optimal point of the problem. Additionally, we establish an iteration complexity bound by showing that the method finds an $\varepsilon$-approximate weakly Pareto point in at most ${\cal O}(1/\varepsilon)$ iterations. Numerical experiments illustrating the practical behavior of the method is presented.
\end{abstract}


\section{Introduction}
 
Multiobjective optimization involves simultaneously minimizing two or more objective functions. Consider the vector-valued function $F:\RR^n \to (\RR\cup\{+\infty\})^m$ defined by $F(x):=(F_1(x),\ldots,F_m(x))$. The associated unconstrained multiobjective optimization problem is denoted as 
\begin{equation}\label{eq:mainprob}
\displaystyle \min_{x \in \RR^n} F(x).
\end{equation}
This paper focuses on problem \eqref{eq:mainprob}, assuming that each component $F_j:\RR^n\to \RR\cup\{+\infty\}$ has the the following special separable structure:
\begin{equation}\label{eq:separable}
F_j(x):= G_j(x)+H_j(x), \quad \forall j=1,\ldots,m,
\end{equation}
where $G_j:\RR^n\to \RR$ is continuously differentiable and convex, and $H_j:\RR^n\to \RR\cup\{+\infty\}$ is proper, convex, and continuous on its domain, but possibly nonsmooth. This particular problem involves many important applications. In particular, when $H_j$ is an indicator function of a convex set ${\cal C}$, \eqref{eq:mainprob} reduces to the convex-constrained problem of minimizing $G(x):=(G_1(x),\ldots,G_m(x))$ subject to $x \in {\cal C}$. Furthermore, the separable structure \eqref{eq:separable} can be used to model robust multiobjective optimization problems. These problems involve dealing with uncertainty in the data, and the optimal solution must occur under the worst possible scenario. We refer the reader to \cite{TanabeFukudaYamashita2019} for more details about applications.

In recent years, a popular strategy for solving multiobjective optimization problems involves extending methods designed for scalar-valued optimization to vector-value optimization. This approach was introduce in the seminal paper \cite{FliegeSvaiter2000}, where the steepest descent method was extended; see also \cite{Drummond2005,doi:10.1080/00036811.2011.640629}. Further advancements in this direction include the extension of the proximal point method proposed in \cite{bonnel2005proximal} and also studied in \cite{Bento2014,Ceng2010, CENG20071, doi:10.1080/01630563.2011.587072, chuong2011hybrid}. Other works exploring various algorithms are presented in \cite{Assuncao2021,Fliege.etall2009, wang2019extended, ansary, Morovati2017, POVALEJ2014765, QU2014503, Goncalves2019, PerezPrudente2018, BELLOCRUZ20115268, FukudaDrummond2011, FukudaDrummond2013,Drummond2004,gonccalves2021globally,Prudente2022}.

Problem \eqref{eq:mainprob} with the separable structure \eqref{eq:separable} in its scalar version (i.e., $m=1$) has been extensively studied; see \cite{BeckTeboulleFista,doi:10.1080/10556788.2016.1214959} and references therein. However, its multiobjective version (i.e., $m\geq 2$) has only recently gained attention. To the best of our knowledge, the first paper addressing this matter is \cite{doi:10.1080/02331934.2018.1440553}, where a forward–backward proximal point-type algorithm was proposed. Proximal-type methods are commonly employed to solve this problem. In \cite{TanabeFukudaYamashita2019}, a proximal gradient method with an Armijo line search was introduced; see also \cite{Tanabe2023rates,chen2023barzilaiborwein}. In \cite{ymsRay}, the authors proposed a proximal gradient method with a line search procedure that uses only information from the differentiable functions $G_j$; however, it may require solving multiple proximal subproblems per iteration. We mention that this approach extends the well-known Beck-Teboulle’s backtracking line search \cite{BeckTeboulleFista} to the vector-valued optimization setting. Accelerated versions of the proximal gradient method were explored in \cite{tanabe2022globally,Tanabe2023accelerated,nishimura2023monotonicity,zhang2023convergence}. Other methods include the condition gradient method \cite{doi:10.1080/02331934.2023.2257709} and Newton-type proximal methods \cite{chen2023convergence,doi:10.1080/10556788.2022.2157000,peng2023quasinewton,peng2022proximal,jiang2023nonmonotone}.

In the present paper, inspired by the work \cite{doi:10.1080/10556788.2016.1214959}, we go a step further than \cite{ymsRay} and introduce a proximal gradient method with an explicit line search procedure characterized by the following features: (i) only one proximal subproblem is solved per iteration; (ii) the backtracking scheme is exclusively applied to the differentiable functions $G_j$. In addition to requiring the solution of only one proximal subproblem, the new line search allows us to explore first-order information of $G_j$, and prevents the computation of potentially costly  functions $H_j$, resulting in computational savings. As discussed in Section~\ref{sec:numerical}, when the structure \eqref{eq:separable} is used to model robust multiobjective optimization problems, the evaluation of $H_j$ often involves solving an auxiliary optimization problem, making our proposal particularly advantageous in such scenarios. Regarding the convergence analysis of the method, we establish, under mild assumptions, the full convergence of the generated sequence to a weakly Pareto optimal solution of the problem. Additionally, we derive an iteration complexity bound by showing that the method finds an {\it $\varepsilon$-approximate weakly Pareto point} in at most ${\cal O}(1/\varepsilon)$ iterations.

This paper is organized as follows. Section~\ref{sec:Preliminares} presents some definitions and basic results used throughout the paper. In Section~\ref{sec:alg}, we introduce our proximal gradient algorithm with the new explicit line search procedure and show that it is well-defined. The asymptotic convergence analysis is presented in Section~\ref{sec:convergence}, and the iteration-complexity bound is established in Section~\ref{sec:complexity}. Numerical experiments are presented in Section~\ref{sec:numerical}. Finally, Section~\ref{sec:concluding} contains a conclusion.


\section{Preliminaries}  \label{sec:Preliminares}

We write $p \coloneqq q$ to indicate that $p$ is defined to be equal to $q$. We denote by $   \NN$ the set of nonnegative integers numbers $\{0, 1, 2,\ldots\}$, by $\RR$ the set of real numbers, and by $\RR_+$  the set of nonnegative real numbers. As usual, $\RR^m$ and $\RR^{m\times n}$ stand for the set of $m$ dimensional real column vectors and the set
of $m\times n$ real matrices, respectively. We define $\overline{\RR}:=\RR \cup \left\{ +\infty\right\}$ and $\overline{\mathbb{R}}^m := (\RR \cup \left\{ +\infty\right\})^m$. For $u, v \in \RR^{m}$, $v\succeq u$ (or $u \preceq v$) means that $v_j \geq u_j$ for all $j=1,\ldots,m$, and $v\succ u$ (or $u \prec v$) means that  $v_j>u_j$ for all $j=1,\ldots,m$. The transpose of the vector $u\in\RR^m$ is denoted by $u^\top$. The Euclidean norm is denoted by $\| \cdot \|$.  

Next we present some definitions and properties for scalar functions. The effective domain of  $\phi:\RR^n \to \overline{\RR}$ is defined as $\dom(\phi):=\left\{ x\in \RR^n \mid \phi(x)< +\infty \right\}$.
The function $\phi$ is said to be {\it proper} if $\dom(\phi)$ is nonempty, and $\phi$ is {\it convex} if,  for every $ x,y \in \dom(\phi)$, there holds
$$\phi(t x + (1 - t)y) \leq t \phi(x) + (1 - t)\phi(y), \quad \forall t \in [0,1].$$
The subdifferential of a proper convex function $\phi$ at $x\in\dom(\phi)$ is defined by
\begin{equation}\label{sub-inq}
\partial \phi(x):=\{u\in \RR^n \mid \phi(y)\ge \phi(x)+ u^\top(y-x) ,\; \forall y\in \RR^n\}.
\end{equation}
  If $\phi: \RR^n \rightarrow \RR$ is convex and differentiable, then the subdiferential is an unitary set, i.e., $\partial \phi(x)= \{ \nabla \phi(x) \}$, and so 
\begin{equation}\label{gradient-inequality}
\phi(y)\geq \phi(x)+\nabla \phi(x)^\top(y-x), \quad \forall x,y\in \RR^n,
\end{equation}
where $\nabla \phi$ denotes the gradient of $\phi$. We say that $\nabla \phi$ is {\it Lipschitz continuous} with constant $L > 0$ if
$$\|\nabla \phi(y) - \nabla \phi(x)\|\leq L\|y-x\|, \quad \forall x,y\in \RR^n.$$
In this case, it holds that
\begin{equation}\label{descentlemma}
    \phi(y)\leq \phi(x)+\nabla \phi(x)^\top(y-x)+\dfrac{L}{2}\|y-x\|^2, \quad \forall x,y\in \RR^n,
\end{equation}
see, for example, \cite[Proposition A.24]{Bertsekas1999}.

Let $F :\mathbb{R}^n \to \overline{\mathbb{R}}^m$ be a vector-valued function given by $F(x):=(F_1(x),\ldots,F_m(x))$. The effective domain and the image of $F$ are denoted by $\dom(F):=\{ x\in \RR^n \mid F_j(x)< +\infty, \forall j=1,\ldots,m \}$ and $\im(F):=\{y\in \RR^m \mid y=F(x), x\in \RR^n\}$, respectively. Since for a multiobjective optimization problem, there is typically no single point minimizing all functions at once, we employ the concept of {\it Pareto optimality} to characterize a solution, as defined below.

\begin{definition}
A point $\bar{x} \in \mathbb{R}^n$ is {\it Pareto optimal} for \eqref{eq:mainprob} if there does not exist another $x \in \mathbb{R}^n$ such that $F(x) \preceq F(\bar{x})$ and $F_i(x) < F_i(\bar{x})$ for at least one index $i \in \{1, \ldots, m\}$. Furthermore, $\bar{x} \in \mathbb{R}^n$ is said to be a {\it weakly Pareto optimal} point for \eqref{eq:mainprob} if there does not exist another $x \in \mathbb{R}^n$ such that $F(x) \prec F(\bar{x})$.
\end{definition}


We conclude this section by recalling the well-known concept of quasi-Fej\'er convergence. 

\begin{definition}\label{def-fejer}
Let $\Omega$ be a nonempty subset of $\RR^n$. A sequence $\{x^k\}$ in $\RR^n$ is said to be quasi-Fej\'er convergent to $\Omega$ if and only if for every $x \in \Omega$ there exists a summable sequence $ \{\varepsilon_k\} \subset \RR_+ $ such that, for every $k\in\NN$, there  holds
\begin{equation*}
\| x^{k+1}-x\|^2 \leq \| x^{k}-x\|^2 +  \varepsilon_k.
\end{equation*}
\end{definition}

The main property of a quasi-Fej\'er convergent sequence is stated in the following lemma, and its proof can be found in  \cite{doi:10.1080/02331939508844042}.

\begin{lemma}\label{lema-Fejer} 
If $\{x^k\}$ is quasi-F\'ejer convergent to $\Omega$, then the following statements hold:
\item {\bf (i)} the sequence $\{x^k\}$ is bounded;
\item {\bf (ii)} if a limit point $x^*$ of $\{x^k\}$ belongs to $\Omega$, then $\{x^k\}$ converges to $x^*$.
\end{lemma}


\section{Algorithm} \label{sec:alg}

For clarity and organization, we explicitly state the assumptions regarding problem \eqref{eq:mainprob}, which will be considered throughout the article.

\vspace{8pt}
\noindent {\bf General Assumption:} {\it The objective function $F:\RR^n \to \overline {\RR}^m$, given by $F(x):=(F_1(x),\ldots,F_m(x))$,  has the following special separable structure:
$$F_j(x):= G_j(x)+H_j(x), \quad \forall j=1,\ldots,m,$$
where:
\begin{itemize}
    \item[(i)] $G_j:\RR^n\to \RR$  is continuously differentiable and convex, for all $j=1,\ldots,m$;
    \item[(ii)] $H_j:\RR^n\to \overline{\RR}$ is proper, convex, and continuous on $\dom(H_j)$, for all $j=1,\ldots,m$;
    \item[(iii)] $\dom(F)$ is nonempty and closed.
\end{itemize}
}

We now introduce a proximal regularization that will be employed in our method.
Given $x\in \dom(F)$ and $\alpha>0$, let us define the function $\psi _{x}:\RR^n\to \overline{\RR}$ as follows:
 \begin{equation}\label{def:psi-theta1}
 \psi _{x} (u) := \max_{j=1,\ldots,m} \left(\nabla G_j(x)^\top(u - x) + H_j(u) - H_j(x) \right), \quad \forall u\in \RR^n.
 \end{equation}
 Now, consider the scalar-valued optimization problem:
 \begin{equation}\label{subprob}
\displaystyle \min_{u\in\RR^n} \psi _{x} (u) + \dfrac{1}{2 \alpha} \| u - x \|^2.
\end{equation}
 Due to the convexity of $H_j$ for all $j=1,\ldots,m$, it follows that $\psi_x$ is convex and the objective function of \eqref{subprob} is strongly convex. Hence, \eqref{subprob} has a unique solution, which trivially belongs to  $\dom(F)$ (noting that $\dom(F) =\dom(\psi_x)$). We denote the solution of \eqref{subprob} by $p_\alpha(x)$ and its optimal value by $\theta_\alpha(x)$, i.e.,
\begin{equation}\label{iterationproxgrad}
p_\alpha(x):=\displaystyle {\rm arg}\min_{u \in \RR^n} \psi _{x} (u) + \dfrac{1}{2 \alpha} \| u - x \|^2
\end{equation}
and 
\begin{equation}\label{def:psi-theta}
 \theta_\alpha(x):=\psi _{x} (p_\alpha(x)) + \dfrac{1}{2 \alpha} \| p_\alpha(x) - x \|^2.
\end{equation}
It is worth noting that when $m=1$, $p_{\alpha}(x)$ is directly related to the well-known {\em forward-backward} operator evaluated at $x$. 

In the following, we state some properties concerning the functions $p_\alpha(\cdot)$ and $\theta_\alpha(\cdot)$, defined in \eqref{iterationproxgrad} and \eqref{def:psi-theta}, respectively.

\begin{lemma}\label{lem:prop-theta}
Given $\alpha>0$, consider $p_\alpha:\dom(F)\to \dom(F)$ and $\theta_\alpha:\dom(F)\to \RR$ as in \eqref{iterationproxgrad} and \eqref{def:psi-theta}, respectively. Then, the following statements are true.
\item[ {\bf (i)}] $\theta_\alpha(x)\leq 0$ for all $x\in \dom(F)$.

\item[ {\bf (ii)}] The following statements are equivalent: (a) $x$ is a weakly Pareto optimal point of \eqref{eq:mainprob}; (b)  $\theta_\alpha(x) = 0$; (c) $p_\alpha(x)=x$.

\item[ {\bf (iii)}] $p_\alpha(\cdot)$ and $\theta_\alpha(\cdot)$ are continuous.
\end{lemma}
\begin{proof}
See\cite[Lemma 3.2]{TanabeFukudaYamashita2019}.
\end{proof}

We now formally describe our proximal gradient method for solving  \eqref{eq:mainprob}.

\vspace{12pt}
	\noindent
	\hrule
	\vspace{0.1cm}
	\noindent
	{\bf Multiobjective Proximal Gradient (MPG) algorithm}
	\vspace{0.1cm}
	\hrule
\begin{description}
\item[\bf Step~0.]  Let $x^0\in \dom(F)$, $\alpha>0$, $\gamma \in (0,2/\alpha)$, and $0<\tau_1<\tau_2<1$ be given. Initialize $k\leftarrow 0$.
\item[\bf Step~1.] {\it Subproblem}\\
Compute $p^k:=p_\alpha(x^k)$ and $\theta_\alpha(x^k)$ as in \eqref{iterationproxgrad} and \eqref{def:psi-theta}, respectively.
\item[\bf Step 2.] \textit{Stopping criterion}\\
		If $\theta_\alpha(x^k)=0$, then STOP.
\item[\bf Step~3.] \textit{Line search procedure}\\
Define $d^k:=p^k-x^k$, take $j_k^* \in \argmax_{j=1,\ldots,m}\nabla G_j(x^k)^\top d^k$, and set $\ttrial =1$.
\begin{description}
\item[\bf Step~3.1.] If 
$$G_{j_k^*}(x^k+\ttrial d^k) \leq G_{j_k^*}(x^k) + \ttrial \nabla G_{j_k^*}(x^k)^\top d^k + \ttrial\frac{\gamma}{2} \| d^k \|^2,$$
then go to Step~3.2. Otherwise, compute $\tnew\in[\tau_1\ttrial,\tau_2\ttrial]$, set $\ttrial\leftarrow \tnew$ and repeat Step~3.1.
\item[\bf Step~3.2.] If 
$$F(x^k+\ttrial d^k)\preceq F(x^k),$$
then define $t_k=\ttrial$ and go to Step~4.
\item[\bf Step~3.3.] Compute $\tnew\in[\tau_1\ttrial,\tau_2\ttrial]$ and set $\ttrial\leftarrow \tnew$. If 
$$G_{j}(x^k+\ttrial d^k) \leq G_{j}(x^k) + \ttrial \nabla G_{j}(x^k)^\top d^k + \ttrial\frac{\gamma}{2} \| d^k \|^2, \quad \forall j=1,\ldots,m,$$
then define $t_k=\ttrial$ and go to Step~4. Otherwise, repeat Step~3.3.
\end{description}
\item[\bf Step~4.] \textit{Iterate}\\
Define $x^{k+1}:=x^k +t_k d^k$, set $k \leftarrow k+1$ and go to Step~$1$.
\nopagebreak
\end{description}\hrule
\nopagebreak
\vspace{8pt}

 Some comments are in order. (a) If $H_j$ lacks smoothness, the subproblem \eqref{subprob} in Step~1 may also lack smoothness. In such instances, the approach for solving \eqref{subprob} depends on the specific structure of the functions $H_j$. In Section~\ref{sec:numerical}, we will explore an application to robust multiobjective optimization and discuss how to solve the subproblem in this particular scenario. (b) At Step~2, it follows from Lemma~\ref{lem:prop-theta}~(ii) that the MPG algorithm stops at iteration $k$ if and only if $x^k$ is a weakly Pareto optimal point. (c) In the line search procedure, the backtracking scheme is applied only to the differentiable functions $G_j$, preventing the computation of potentially costly $H_j$ functions. In Step~3.1, we first perform a line search using a single function $G_{j_k^*}$. If the resulting $\ttrial$ step size leads to a decrease in all $F_j$'s, it is accepted and the line search is finished. Otherwise, the line search is carried out on all $G_j$'s functions until the condition in Step~3.3 is satisfied. Thus, the $k$-th iteration concludes with one of the following scenarios:
 \begin{equation} \label{LS1}
    G_{j_k^*}(x^{k+1}) \leq G_{j_k^*}(x^k) + t_k \nabla G_{j_k^*}(x^k)^\top d^k + t_k\frac{\gamma}{2} \| d^k \|^2 \; \mbox{and} \; \left[F(x^{k+1})\preceq F(x^k)\right],
\end{equation}
or
\begin{equation} \label{LS2}
G_{j}(x^{k+1}) \leq G_{j}(x^k) + t_k \nabla G_{j}(x^k)^\top d^k + t_k\frac{\gamma}{2} \| d^k \|^2, \quad \forall j=1,\ldots,m.
\end{equation}
It is worth mentioning that, given $j\in\{1,\ldots,m\}$, $d^k$ is not necessarily a descent direction for $G_j$ at $x^k$, meaning that it can happen that $\nabla G_{j}(x^k)^\top d^k\geq 0$.
However, as we will see, both cases \eqref{LS1} and \eqref{LS2} result in a decrease in the value of each objective $F_j$.

In the following, we show that the MPG algorithm is well defined. This means that if the algorithm does not stop at $x^k$, then it is possible to obtain $x^{k+1}$ in a finite time.

\begin{theorem}
The MPG algorithm is well defined and stops at iteration $k$ if and only if $x^k$ is a weakly Pareto optimal point.
\end{theorem}
\begin{proof}
Due to the strong convexity of the objective function in \eqref{subprob}, the subproblem at Step~1 is solvable, allowing the computation of $p^k$ and $\theta_\alpha(x^k)$. According to Lemma~\ref{lem:prop-theta}~(ii), the MPG algorithm stops at Step~2 in iteration $k$ if and only if $x^k$ is a weakly Pareto optimal point. Now, assuming that $x^k$ is not a weakly Pareto optimal point, it follows from Lemma~\ref{lem:prop-theta}~(ii) that $d^k\neq 0$. Thus, for any arbitrary index $j\in\{1,\ldots,m\}$, the differentiability of $G_j$ ensures the existence of $\delta>0$ such that
$$\left|\dfrac{G_{j}(x^k+t d^k) -G_{j}(x^k)}{t} - \nabla G_{j}(x^k)^\top d^k\right|  \leq \frac{\gamma}{2} \| d^k \|^2,$$
for all $t\in(0,\delta]$. Hence,
$$G_{j}(x^k+t d^k) \leq G_{j}(x^k) + t \nabla G_{j}(x^k)^\top d^k + t\frac{\gamma}{2} \| d^k \|^2, \quad \forall j=1,\ldots,m,$$
for all $t\in(0,\delta]$. Therefore, the line search procedure ultimately finishes in a finite number of (inner) steps, and $x^{k+1}$ is properly defined at Step~4.
\end{proof}

\section{Asymptotic convergence analysis} \label{sec:convergence}

The MPG algorithm successfully stops if a weakly Pareto optimal point is found. Then, in order to analyze its convergence properties, we assume henceforth that the MPG algorithm generates an infinite sequence, which is equivalent to say that none $x^k$ is a weakly Pareto optimal point of problem \eqref{eq:mainprob}. 





In the following, we establish some key inequalities for our analysis. In particular, we show that if, at iteration $k$, the backtracking in the line search procedure is based on $G_j$, it automatically leads to a decrease in the corresponding objective $F_j$. 

\begin{lemma} \label{lema-para-qF} 
Let $ \{x^k\} $ be generated by the MPG algorithm. Suppose that at iteration $k$, the index $j\in\{1,\ldots,m\}$ is such that
\begin{equation}\label{ineq:Linesearch-VPGM}
G_{j}(x^{k+1}) \leq  G_{j}(x^k) + t_k \nabla G_{j}(x^k)^\top d^k + t_ k\frac{\gamma}{2} \| d^k \|^2.
\end{equation}
Then
\item[ {\bf (i)}] $\psi_{x^k}(p^k) \geq \dfrac{1}{t_k} \left(  F_j(x^{k+1}) - F_j(x^k)\right)  - \dfrac{\gamma}{2} \| d^k \|^2$;
\item[ {\bf (ii)}] for every $x\in \dom(F)$, we have 
\begin{equation}\label{argentina}
\begin{split}
\|x^{k+1}-x\|^2 \leq \; & \|x^{k}-x\|^2 + 2\alpha  \left(  F_j(x^k) - F_j(x^{k+1})\right) +2 \alpha t_k \max_{i=1,\ldots,m} \left(  F_i(x) - F_i(x^k) \right) \\
                & - 2\alpha t_k \left(\dfrac{1}{\alpha}-\frac{\gamma}{2}\right)\|d^k\|^2 +t_k^2\|d^k\|^2;
 \end{split}
 \end{equation}
\item[ {\bf (iii)}] $F_j(x^{k+1}) - F_j(x^{k}) \leq     - t_k \left(\dfrac{1}{\alpha}-\dfrac{\gamma}{2}\right)\|d^k\|^2$.
\end{lemma}
\begin{proof}
(i) It follows from  \eqref{ineq:Linesearch-VPGM} that 
$$\nabla G_{j}(x^k)^\top d^k \geq \frac{1}{t_k} \left( G_{j}(x^{k+1}) - G_{j}(x^k) \right) - \frac{\gamma}{2} \| d^k \|^2.$$
Hence, in view of the definition of $ \psi_{x^k} $ in \eqref{def:psi-theta1}, we have
\begin{align*}
\psi_{x^k}(p^k) & \geq   \nabla G_j(x^k)^\top d^k + H_j(p^k) - H_j(x^k) \nonumber \\
& \geq  \frac{1}{t_k} \left[ G_{j}(x^{k+1}) - G_{j}(x^k) + t_k \left(H_j(p^k) - H_j(x^k)\right)\right] - \frac{\gamma}{2} \| d^k \|^2. 
\end{align*}
Since $x^{k+1}=x^k +t_k (p^k -x^k)$ with  $t_k\in (0, 1]$, by the convexity of $H_j$, we have 
$$
t_k \left(H_j(p^k) - H_j(x^k)\right) \geq H_j(x^{k+1}) - H_j(x^k).
$$
By combining the latter two inequalities and using that $F_j=G_j+H_j$, we obtain the desired result.

(ii) Let  $x\in \dom(F)$. In view of the definition of $x^{k+1}$ in Step~4, we have 
\begin{align}
\|x^{k+1}-x\|^2 & = \|x^{k}-x\|^2 +\|x^{k+1}-x^k\|^2+2\la x^{k+1}-x^{k},x^k-x\ra\nonumber\\
&=\|x^{k}-x\|^2 +t_k^2\|d^k\|^2+2 t_k\la d^{k},x^k-x\ra. \label{eq:xk-x}
\end{align}
Our goal now is to appropriately estimate the quantity $2t_k\la d^{k},x^k-x\ra$. 
The first-order optimality condition of \eqref{iterationproxgrad} implies that
$-(d^k/\alpha) \in \partial \psi_{x^k}(p^k)$. Hence, by the subgradient inequality \eqref{sub-inq} and using item (i), we have 
\begin{align}
\psi_{x^k}(x) & \geq \psi_{x^k}(p^k) + \dfrac{1}{\alpha} \langle d^k,p^k-x \rangle \nonumber \\
& = \psi_{x^k}(p^k)  + \dfrac{1}{\alpha}\left\langle d^k, x^k - x \right\rangle+\dfrac{1}{\alpha}\|d^k\|^2 \nonumber \\
& \geq\frac{1}{t_k} \left(  F_j(x^{k+1}) - F_j(x^k)\right)  + \dfrac{1}{\alpha}\left\langle d^k, x^k - x \right\rangle+\left(\dfrac{1}{\alpha}-\frac{\gamma}{2}\right)\|d^k\|^2 \label{brasil}.
\end{align}
Now, in view of the definition of $\psi_{x}$ in \eqref{def:psi-theta1}, the gradient inequality \eqref{gradient-inequality} with $ \phi = G_i $, and the fact that $F_i=G_i+H_i$, $ i=1,\ldots,m, $ we have   $\displaystyle\max_{i=1,\ldots,m} \left(  F_i(x) - F_i(x^k) \right) \geq \psi_{x^k}(x)$, which combined with \eqref{brasil} imply that
\begin{equation*}
2t_k\la d^{k},x^k-x\ra \leq  2\alpha  \left(  F_j(x^k) - F_j(x^{k+1})\right) + 2\alpha t_k \max_{i=1,\ldots,m} \left(  F_i(x) - F_i(x^k) \right)  - 2 t_k \left(1-\frac{\alpha \gamma}{2}\right)\|d^k\|^2.
\end{equation*}
This inequality together with \eqref{eq:xk-x} yields \eqref{argentina}.

(iii) By setting $x=x^k$ in \eqref{argentina} and taking into account that $\|x^{k+1}-x^k\|^2=t_k^2\|d^k\|^2$, some algebraic manipulations give the desired inequality.
\end{proof}

\begin{remark}\label{remark}
Given that, at iteration $k$, the line search procedure terminates with a step size $t_k$ satisfying either \eqref{LS1} or \eqref{LS2}, the inequality \eqref{ineq:Linesearch-VPGM} holds for $j=j_k^* \in \argmax_{j=1,\ldots,m}\nabla G_j(x^k)^\top d^k$. Consequently, Lemma~\ref{lema-para-qF} is always applicable to the index $j_k^*$.
\end{remark}

In the following lemma, we show that the objective function sequence $\{F_j(x^k)\}$ is nonincreasing for each $j=1,\ldots,m$. 

\begin{lemma}\label{uruguai}
Let $ \{x^k\} $ be generated by the MPG algorithm. Then, for all $j=1,\ldots,m$, it holds that $\{F_j(x^k)\}$ is a non-increasing sequence, i.e., 
$$F(x^{k+1})\preceq F(x^{k}), \quad  \forall k\in\NN.$$
\end{lemma}
\begin{proof}
Let $k\in\NN$. We consider two cases: (i) \eqref{LS1} holds at iteration $k$; (ii) \eqref{LS2} holds at iteration $k$. In the first case, then statement of the lemma is a straightforward consequence of the second condition in \eqref{LS1}. Consider now the second case. Taking into account that $\gamma<2/\alpha$, it follows from Lemma~\ref{lema-para-qF}(iii) that $F_j(x^{k+1})< F_j(x^{k})$ for all $j=1,\ldots,m$, concluding the proof.
\end{proof}

To establish the convergence of the MPG algorithm, we introduce the following additional assumption, meaning that the set $\im(F)$ is {\it complete} with respect to the {\it Paretian} cone $\RR^m_+:=\{x\in\RR^m\mid x\succeq 0\}$.



\vspace{8pt}
\noindent {\bf Assumption (A1):} {\it All monotonically nonincreasing sequences in the set $\im(F)$ are bounded below by a point in $\im(F)$, i.e., for every sequence $\{y^k\}\subset \dom(F)$ such that $F(y^{k+1}) \preceq F(y^k)$ for all $k\in\NN$, there exists  $y\in\dom(F)$ such that $F(y) \preceq F(y^k)$ for all $k\in\NN$.}
\vspace{8pt}

Assumption~(A1) is widely used in the related literature (e.g., \cite{doi:10.1080/02331934.2018.1440553,Drummond2005,bonnel2005proximal,Drummond2004,Bento2014,ymsRay,BELLOCRUZ20115268,FukudaDrummond2013,Ceng2010,CENG20071,doi:10.1080/01630563.2011.587072,doi:10.1080/00036811.2011.640629}).
It is commonly employed to ensure the existence of efficient solutions for vector-valued optimization problems, see \cite[Section~2.3]{DinhTheLuc}. In the scalar-valued unconstrained case, it is equivalent to ensuring the existence of an optimal point.


\begin{theorem}\label{teo:conv}
Let $\{x^k\}$ be generated by the MPG algorithm and suppose that Assumption~(A1) holds. Then, $\{x^k\} $ converges to a weakly Pareto optimal point $\bar{x}$ of problem~\eqref{eq:mainprob}.
\end{theorem}
\begin{proof} 
First, we define
$$\Omega := \{ x \in \dom(F) \mid F(x) \preceq F(x^k), \; \forall k\in\NN\}.$$
In view of Lemma~\ref{uruguai}, it follows from Assumption~(A1) that $\Omega$ is nonempty. Thus, let us consider an arbitrary $\hat{x} \in \Omega$, i.e., $F(\hat{x}) \preceq F(x^k)$ for all $k\in\NN$. Therefore, taking into account Remark~\ref{remark} and that $\gamma < 2/\alpha$, applying Lemma~\ref{lema-para-qF}(ii) with $j=j_k^*$ and $x=\hat{x}$ yields
\begin{equation}\label{aux:quasi-fejer-VPGM}
\|x^{k+1}-\hat{x}\|^2 \leq \|x^{k}-\hat{x}\|^2 + \varepsilon_k, \quad \forall k\in\NN.
\end{equation}
where
\begin{equation}\label{epsilonk}
\varepsilon_k:= 2\alpha  \left(  F_{j_k^*}(x^k) - F_{j_k^*}(x^{k+1})\right) +t_k^2\|d^k\|^2, \quad \forall k\in\NN.
\end{equation}
Note that Lemma~\ref{uruguai} implies that $\varepsilon_k\geq 0$ for every $k \in\NN$. Moreover, since $t_k^2\leq t_k$, we obtain from Lemma~\ref{lema-para-qF}(iii)  with $j=j_k^*$ that
$$t_k^2 \left(\dfrac{1}{\alpha}-\frac{\gamma}{2}\right)\|d^k\|^2\leq t_k \left(\dfrac{1}{\alpha}-\frac{\gamma}{2}\right)\|d^k\|^2\leq  F_{j_k^*}(x^{k})- F_{j_k^*}(x^{k+1}), \quad \forall k\in\NN,$$
which yields 
    $$t_k^2 \|d^k\|^2\leq  \dfrac{2\alpha}{2-\alpha \gamma}\left( F_{j_k^*}(x^{k})- F_{j_k^*}(x^{k+1})\right), \quad \forall k\in\NN.$$
Hence, by \eqref{epsilonk} and Lemma~\ref{uruguai}, we have
$$\varepsilon_k \leq c\left(F_{j_k^*}(x^{k})- F_{j_k^*}(x^{k+1})\right)\leq c\sum_{j=1}^m \left(F_j(x^{k})- F_j(x^{k+1})\right), \quad \forall k\in\NN,$$
where $c:=2\alpha + 2\alpha/(2-\alpha\gamma)>0$.
By summing this expression over all indices less than or equal to $N\in\NN$, and taking into account that $\hat{x} \in \Omega$, we obtain
$$\sum_{k=0}^N \varepsilon_k \leq c\sum_{j=1}^m \sum_{k=0}^N\left(F_j(x^{k})- F_j(x^{k+1})\right)=c\sum_{j=1}^m \left(F_j(x^{0})- F_j(x^{N+1})\right)\leq c\sum_{i=j}^m \left(F_j(x^{0})- F_j(\hat{x})\right)<\infty.$$
Therefore, $\sum_{k=0}^{+ \infty} \varepsilon_k < \infty$. Hence, it follows from \eqref{aux:quasi-fejer-VPGM} that $ \{ x^k\} $ is quasi-Fej\'er convergent to $ \Omega $, see Definition~\ref{def-fejer}. 
From Lemma~\ref{lema-Fejer}(i), it follows that $ \{x^k\} $ is bounded. Let $\bar{x}$ be a limit point of $ \{x^k\}$. Since $\dom(F)$ is closed, $F$ is continuous on $\dom(F)$, and $F(x^{k+1})\preceq F(x^k)$ for all $k\in\NN$, we have $\bar{x}\in\Omega$. Thus, Lemma~\ref{lema-Fejer}(ii) implies that the whole sequence $ \{x^k\} $ converges to $\bar{x}$. 

Let us now show that $\bar{x}$ is a weakly Pareto optimal point of problem~\eqref{eq:mainprob}. Since  $t_k\|d^k\|=\|x^{k+1}-x^{k}\|$ and $\{x^k\}$ converges to $\bar{x}$, we obtain
\begin{equation}\label{equador}
    \lim_{k\to\infty}t_k\|d^k\|=0.
\end{equation}
We now  claim that there exists a subsequence $\{d^{k_\ell}\}$ such that $\lim_{\ell\to\infty}\|d^{k_\ell}\|=0$.
Indeed, if $\limsup_{k\to\infty} t_k>0$, then the claim clearly holds from \eqref{equador}. Hence, assume that $\limsup_{k\to\infty} t_k=0$, which in turn implies that $\lim_{k\to\infty} t_k=0$. 
Without loss of generality, suppose that $t_k<1$ for all $k\in\NN$. Therefore, by Step~3 of the algorithm and taking into account Lemma~\ref{lema-para-qF}(iii), for each $k\in\NN$ there exist at least one index $i_k\in\{1,\ldots,m\}$ and 
\begin{equation}\label{peru}
    0<\bar{t}_k \leq \frac{t_k}{\tau_1}
\end{equation}
such that
$$G_{i_k}(x^k+\bar{t}_kd^k) > G_{i_k}(x^k) + \bar{t}_k \nabla G_{i_k}(x^k)^\top d^k + \bar{t}_k\frac{\gamma}{2} \| d^k \|^2.$$
Since $\{1,\ldots,m\}$ is finite, there exist $i_*\in \{1,\ldots,m\}$ and an infinite set of indices $\{k_\ell\}\subset\NN$ such that
$$G_{i_*}(x^{k_\ell}+\bar{t}_{k_\ell}d^{k_\ell}) - G_{i_*}(x^{k_\ell}) > \bar{t}_{k_\ell} \nabla G_{i_*}(x^{k_\ell})^\top d^{k_\ell} + \bar{t}_{k_\ell}\frac{\gamma}{2} \| d^{k_\ell} \|^2, \quad \forall \ell\in\NN.$$
On the other hand, since $G_{i_*}$ is convex, it follows from \eqref{gradient-inequality} that
$$G_{i_*}(x^{k_\ell}+\bar{t}_{k_\ell}d^{k_\ell}) - G_{i_*}(x^{k_\ell}) \leq \bar{t}_{k_\ell} \nabla G_{i_*}(x^{k_\ell}+\bar{t}_{k_\ell}d^{k_\ell})^\top d^{k_\ell}, \quad \forall \ell\in\NN.$$
By combining these two inequalities, we obtain
$$\dfrac{\gamma}{2} \| d^{k_\ell} \|^2 <  \left(\nabla G_{i_*}(x^{k_\ell}+\bar{t}_{k_\ell}d^{k_\ell})-\nabla G_{i_*}(x^{k_\ell})\right)^\top d^{k_\ell}, \quad \forall \ell\in\NN.$$
Therefore, by the Cauchy-Schwarz inequality, we have
$$ \| d^{k_\ell} \| < \frac{2}{\gamma} \left\|\nabla G_{i_*}(x^{k_\ell}+\bar{t}_{k_\ell}d^{k_\ell})-\nabla G_{i_*}(x^{k_\ell})\right\|, \quad \forall \ell\in\NN.$$
Note that \eqref{equador} and \eqref{peru} imply that $\lim_{\ell\to\infty}\bar{t}_{k_\ell}d^{k_\ell}=0$. Thus, since $G_{i_*}$ is continuously differentiable, by taking limits as $\ell\to\infty$ in the latter inequality, we conclude that $\lim_{\ell\to\infty}\|d^{k_\ell}\|=0$, as claimed. Hence, it follows from the definition of $d^k$ and Lemma~\ref{lem:prop-theta}(iii) that 
$$0=\lim_{\ell\to\infty}\|d^{k_\ell}\|=\lim_{\ell\to\infty}\|p_\alpha(x^{k_\ell})-x^{k_\ell}\|=\|p_\alpha(\bar{x})-\bar{x}\|.$$
Therefore, $p_\alpha(\bar{x})=\bar{x}$, and then, in view of Lemma~\ref{lem:prop-theta}(ii), we conclude that $\bar{x}$ is a weakly Pareto optimal point of problem \eqref{eq:mainprob}.
\end{proof}

\section{Iteration-complexity bound} \label{sec:complexity}

In this section, we establish an iteration-complexity bound for the MPG algorithm to obtain an approximate weakly Pareto optimal point of \eqref{eq:mainprob}. This means that, given a tolerance $\varepsilon>0$, we aim to measure the maximum number of iterations required to find a point $x^k$ satisfying $|\theta_\alpha(x^k)|\leq \varepsilon$. For that, we need the following additional assumption.

\vspace{8pt}
\noindent {\bf Assumption (A2):} {\it The gradient $\nabla G_j$ is Lipschitz continuous with constant $L_j>0$, for all $j=1,\ldots,m$.}
\vspace{8pt}

We start our analysis by showing that, under Assumption~(A2), the step size $t_k$ remains bounded away from zero.

\begin{lemma} \label{lem:tmin}
Let $\{x^k\}$ be generated by the MPG algorithm and suppose that Assumption~(A2) holds.
Then, for all $k\in\NN$, we have $t_k\geq t_{\min}:=\min\{1,\tau_1 \gamma/L_{\max}\}$, where $L_{\max}:=\max_{j=1,\ldots,m}L_j$.
\end{lemma}
\begin{proof}
Let $k\in\NN$. If $t_k=1$, then the required inequality trivially holds.
Thus, assume that $t_k<1$.
Therefore, by Step~3 of the algorithm and taking into account Lemma~\ref{lema-para-qF}(iii), there exist at least one index $i_k\in\{1,\ldots,m\}$ and $0<\bar{t}_k \leq t_k/\tau_1$
such that
$$G_{i_k}(x^k+\bar{t}_kd^k) > G_{i_k}(x^k) + \bar{t}_k \nabla G_{i_k}(x^k)^\top d^k + \bar{t}_k\frac{\gamma}{2} \| d^k \|^2.$$
On the other hand, considering Assumption~(A2), it follows from \eqref{descentlemma} that
$$G_{i_k}(x^k+\bar{t}_kd^k) \leq G_{i_k}(x^k) + \bar{t}_k \nabla G_{i_k}(x^k)^\top d^k + \bar{t}_k^2\frac{L_{\max}}{2} \| d^k \|^2.$$
By combining the two latter inequalities, we easily obtain that $\bar{t}_k  > \gamma/L_{\max}$. Hence, by the definition of $\bar{t}_k$, we have
$$t_k>\dfrac{\tau_1 \gamma}{L_{\max}},$$
and the proof is concluded.
\end{proof}

The computational cost associated with each iteration $k$ includes solving the subproblem at Step~1 and computing the step size $t_k$ at Step~3. In terms of the line search procedure, the computational cost essentially involves evaluations of $G_j$ functions at different trial points. The following remark outlines the maximum number of function evaluations needed to compute $t_k$.

\begin{remark}
For iteration $k$, let $\omega(k) \geq 0$ denote the number of backtracking iterations in the line search procedure to calculate $t_k$, i.e., the number of $\tnew$ step sizes computed. Thus, $t_k\leq \tau_2^{\omega(k)}$. Since, by Lemma~\ref{lem:tmin}, $t_k\geq t_{\min}$, we obtain $\log(t_{\min})\leq \omega(k) \log(\tau_2)$. Considering that $\tau_2<1$, it follows that $\omega(k)\leq \log(t_{\min})/\log(\tau_2)$. Consequently, the line search procedure involves at most:
\begin{itemize}
    \item $1+m\omega(k) \leq 1+ m \log(t_{\min})/\log(\tau_2)$ evaluations of $G_j$'s functions at Steps~3.1 and 3.3;
    \item $m-1$ evaluations of $H_j$'s functions at Step~3.2.
\end{itemize}
\end{remark}

The following auxiliary result establishes a lower bound on the optimal value $\theta_\alpha(x^k)$.

\begin{lemma} \label{lem:thetabound}
Let $\{x^k\}$ be generated by the MPG algorithm and suppose that Assumption~(A2) holds.
Then, there exists $c>0$ such that
$$\theta_\alpha(x^k) \geq c \left(  F_{j_k^*}(x^{k+1}) - F_{j_k^*}(x^k)\right), \quad \forall k\in\NN,$$
where $j_k^* \in \argmax_{j=1,\ldots,m}\nabla G_j(x^k)^\top d^k$.
\end{lemma}
\begin{proof}
 From the definition of $\theta_\alpha(x^k)$ in \eqref{def:psi-theta}, Lemma~\ref{lema-para-qF}(i) with $j=j_k^*$, and taking into account that $\gamma<2/\alpha$, we obtain
\begin{align*}
\theta_\alpha(x^k) & \geq \dfrac{1}{t_k} \left(  F_{j_k^*}(x^{k+1}) - F_{j_k^*}(x^k)\right)  + \left(\dfrac{1}{2\alpha}-\dfrac{\gamma}{2} \right)\| d^k \|^2 \nonumber \\
&   \geq  \frac{1}{t_k} \left(  F_{j_k^*}(x^{k+1}) - F_{j_k^*}(x^k)\right)  -  \frac{1}{2\alpha} \| d^k \|^2, \quad \forall k\in\NN.
\end{align*}
Therefore, it follows from Lemma~\ref{lema-para-qF}(iii) with $j=j_k^*$ that
$$\theta_\alpha(x^k) \geq \left(1+\dfrac{1}{2-\gamma\alpha} \right) \frac{1}{t_k} \left(  F_{j_k^*}(x^{k+1}) - F_{j_k^*}(x^k)\right), \quad \forall k\in\NN.$$
Hence, by Lemma~\ref{lem:tmin} and taking into account that $F_{j_k^*}(x^{k+1}) - F_{j_k^*}(x^k)<0$, we have 
$$\theta_\alpha(x^k) \geq \left(1+\dfrac{1}{2-\gamma\alpha} \right) \frac{1}{t_{\min}} \left(  F_{j_k^*}(x^{k+1}) - F_{j_k^*}(x^k)\right), \quad \forall k\in\NN.$$
By defining $c:=[1+1/(2-\gamma\alpha)]/t_{\min}$, we conclude the proof.
\end{proof}

We are now able to establish the iteration-complexity bound for the MPG algorithm.

\begin{theorem}
Suppose that Assumptions~(A1) and (A2) hold.
Then, for a given tolerance $\varepsilon>0$, the MPG algorithm generates a point $x^k$ such that $|\theta_\alpha(x^k)|\leq \varepsilon$ in at most ${\cal O}(1/\varepsilon)$ iterations.
\end{theorem}
\begin{proof}
Assume that none of the generated points $x^k$, $k=0,\ldots,N$, is an $\varepsilon$-approximate weakly Pareto optimal point of \eqref{eq:mainprob}, i.e., $|\theta_\alpha(x^k)|> \varepsilon$. Thus, from Lemma~\ref{lem:thetabound}, we obtain
$$(N+1)\varepsilon < (N+1) \min\{|\theta_\alpha(x^k)| \mid k=0,\ldots,N\} \leq \sum_{k=0}^N|\theta_\alpha(x^k)|\leq c \sum_{k=0}^N \left( F_{j_k^*}(x^k) - F_{j_k^*}(x^{k+1}) \right).$$
Hence, it follows from Lemma~\ref{uruguai} that
$$(N+1)\varepsilon < c \sum_{j=1}^m\sum_{k=0}^N \left( F_j(x^k) - F_j(x^{k+1}) \right)=c\sum_{j=1}^m \left(F_j(x^{0})- F_j(x^{N+1})\right)\leq c\sum_{j=1}^m \left(F_j(x^{0})- F_j(\bar{x})\right),$$
where $\bar{x}$ is as in Theorem~\ref{teo:conv}. Therefore,
$$N+1 < \dfrac{c\sum_{i=j}^m \left(F_j(x^{0})- F_j(\bar{x})\right)}{\varepsilon},$$
concluding the proof.
\end{proof}

\section{Numerical Results}\label{sec:numerical}

This section provides numerical results to assess the practical performance of the MPG algorithm. We are mainly interested in verifying the effectiveness of using the new explicit line search procedure in a proximal-gradient-type method. For this purpose, we consider the following methods in the reported tests:
\begin{itemize}
    \item MPG: The multiobjective Proximal Gradient algorithm with the explicit line search procedure proposed in Section~\ref{sec:alg}.
    \item MPG-Armijo: The {\it vanilla} Proximal Gradient algorithm with Armijo line search proposed in \cite{TanabeFukudaYamashita2019}. The step size is defined as
    $$t_k=:\max_{\ell\in\NN}\{2^{-\ell} \mid F_j( x^k+ 2^{-\ell}   d^k)\leq  F_j(x^k) - \sigma 2^{-\ell}  \psi _{x^k} (p^k), \; \forall j=1,\ldots,m\},$$
    where $\sigma:=10^{-4}$ is an algorithmic parameter.
    \item MPG-Implicit: The {\it implicit} Proximal Gradient algorithm proposed in \cite{ymsRay}. This algorithm can be viewed as a precursor to the MPG algorithm. It is similar to the MPG algorithm, with the line search procedure replaced by the following scheme:
    If
$$G_{j}(p_\alpha(x^k)) \leq G_{j}(x^k) + \nabla G_{j}(x^k)^\top d^k + \frac{1}{2\alpha} \| d^k \|^2, \quad \forall j=1,\ldots,m,$$
then $x^{k+1}:=p_\alpha(x^k)$, and a new iteration is initialized. Otherwise, set $\alpha\leftarrow\alpha/2$ and solve the proximal subproblem \eqref{iterationproxgrad} again.
\end{itemize}

The algorithmic parameters for the MPG implementation are set as follows: $\alpha = 1$, $\gamma = 1.9999$, $\tau_1=0.1$, and $\tau_2=0.9$. Given that the explicit line search procedure utilizes only information from the differentiable functions $G_j$, we exploit this characteristic by implementing the backtracking scheme based on quadratic polynomial interpolations of $G_j$'s functions.
The success stopping criterion for all algorithms is defined as:
\begin{equation}\label{eq:stop}
    |\theta_\alpha(x^k)|\leq 10^{-4}.
\end{equation}
The maximum number of allowed iterations is set to 200, after which the algorithm is considered to have failed.
The codes are written in Matlab and are freely available at \url{https://github.com/lfprudente}.

The chosen test problems are related to robust multiobjective optimization (see \cite{doi:10.1137/080734510}). Robust optimization deals with uncertainty in the data of optimization problems, requiring the optimal solution to occur in the worst possible scenario, i.e., for the most adverse values that the uncertain data may take.



In the following,  we discuss how the test problems are designed (see \cite{doi:10.1080/02331934.2023.2257709}). The differentiable components $G_j$ come from convex multiobjective problems found in the literature. Table~\ref{tab:problems} provides the main characteristics of the selected problems, including the problem name, the corresponding reference for its formulation, the number of variables ($n$), and the number of objectives ($m$). For each test problem, we assume that  $\mathrm{dom}(H_j)=\{x\in\RR^n \mid lb \preceq x\preceq ub\}$ for all $j=1,\ldots,m$, where $lb,ub\in\RR^n$ are given in the last columns of the table. Thus, $\dom(F)$ also coincides with this box.

\begin{table}[h!]
\centering
\begin{threeparttable}{\fontsize{7}{7.1}\selectfont
\rowcolors{2}{}{lightgray}
\begin{tabular}{|cccccc|} 
\hline
\rowcolor[gray]{.90} Problem & Ref. & $n$  & $m$  & $lb$ & $ub$ \\ 
\hline              
AP1&  \cite{ansary} &     2 &  3   & $(-10,-10)$ & $(10,10)$\\
AP2&  \cite{ansary} &     1 &  2   & $-100$ & $100$\\
AP4&  \cite{ansary}  &     3 &  3   & $(-10,-10,-10)$ & $(10,10,10)$\\
BK1&  \cite{reviewproblems}  &  2 &  2  & $(-5,-5)$  & $(10,10)$   \\ 
DGO2&  \cite{reviewproblems}  &  1 &  2   & $-9$  & $9$   \\ 
FDS &  \cite{Fliege.etall2009}  & 5 & 3 & $(-2,\ldots,-2)$ &  $(2,\ldots,2)$  \\  
IKK1&  \cite{reviewproblems}  &  2 &  3   & $(-50,-50)$  & $(50,50)$   \\ 
JOS1&  \cite{10.5555/2955239.2955427} &  100 &  2   & $(-100,\ldots,-100)$ & $(100,\ldots,100)$\\
Lov1& \cite{doi:10.1137/100784746} & 2 & 2  & $(-10,-10)$ & $(10,10)$\\
MGH33 & \cite{ellen} &    10 &  10   & $(-1,\ldots,-1)$ &  $(1,\ldots,1)$ \\ 
MHHM2&  \cite{reviewproblems}  &  2 &  3   & $(0,0)$  & $(1,1)$   \\ 
MOP7&  \cite{reviewproblems} &     2 &  3   & $(-400,-400)$  & $(400,400)$  \\ 
PNR&  \cite{pnr} &  2 &  2  & $(-2,-2)$  &   $(2,2)$\\ 
SD&  \cite{stadler1993multicriteria} &  4 &  2   & $(1,\sqrt{2},\sqrt{2},1)$  & $(3,3,3,3)$   \\ 
SLCDT2&  \cite{slcdt} &   10 &  3   & $(-1,\ldots,-1)$  &  $(1,\ldots,1)$ \\ 
SP1&  \cite{reviewproblems} &     2 &  2   &  $(-100,-100)$ & $(100,100)$  \\ 
Toi4&  \cite{ellen} &  4 &  2   & $(-2,-2,-2,-2)$  & $(5,5,5,5)$  \\ 
Toi8&  \cite{ellen} &  3 &  3   & $(-1,-1,-1,-1)$  & $(1,1,1,1)$  \\ 
VU2&  \cite{reviewproblems} &   2 &  2   & $(-3,-3)$ &$(3,3)$   \\ 
ZDT1 & \cite{doi:10.1162/106365600568202} & 30 & 2  & $(0,\ldots,0)$ & $(1,\ldots,1)$ \\ 
ZLT1&  \cite{reviewproblems} &   10 &  5   & $(-1000,\ldots,-1000)$ &$(1000,\ldots,1000)$   \\ 
\hline
\end{tabular}}
\end{threeparttable}
\caption{List of test problems.}
\label{tab:problems}
\end{table}

For a given problem in Table~\ref{tab:problems}, we denote the {\it uncertainty parameter} by $z\in\RR^n$, and assume that the objective functions are as follows:
$$F_j(x):= G_j(x)+  x^\top z, \quad x\in\dom(F), \quad \forall j=1,\ldots,m.$$
 Minimizing $F_j(x)$ under the worst-case scenario involves solving
 $$\min_{x\in \dom(F)} G_j(x)+ \max_{z\in {\cal Z}_j} x^\top z,$$
 where ${\cal{Z}}_{j} \subset \RR^n$ is the {\it uncertainty set}. 
 Thus, we define $H_j:\dom(F)\to\RR$ as
 \begin{equation}\label{hdef}
H_{j}(x):=\displaystyle \max_{z \in {\cal{Z}}_j}   x^\top z, \quad \forall j=1,\ldots,m.
\end{equation}
Assuming that ${\cal{Z}}_{j}$ is a nonempty and bounded polyhedron given by ${\cal{Z}}_{j}=\{z\in\RR^n \mid A_jz \preceq b_j\}$, where $A_j\in\RR^{d\times n}$ and $b_j\in\RR^d$, we can express \eqref{hdef} as
\begin{equation} \label{evalh}
\begin{array}{cl}
\max_{z}   &  x^\top z         \\
\mbox{s.t.} & A_j z \preceq b_j. \\
\end{array}
\end{equation}
This means that evaluating $H_j(\cdot)$ requires solving a linear programming problem.

Let us discuss how the subproblem can be solved. It is easy to see that \eqref{iterationproxgrad} can be reformulated by introducing an extra variable $\tau\in \RR$ as follows
\begin{equation} \label{vproblem}
 \begin{array}{cl}
\min_{\tau,u}   &  \tau + \frac{1}{2 \alpha} \| u - x^k \|^2         \\
\mbox{s.t.} & \nabla  G_j(x^k)^\top(u-x^k) + H_{j}(u)-H_{j}(x^k)\leq \tau,  \quad \forall j=1,\ldots,m,\\
                            & lb\preceq u\preceq ub.
\end{array}
\end{equation}
Now, by noting that the dual problem of \eqref{evalh} is given by
\[\begin{array}{cl}
\min_{w_j}   &  b_j^\top w_j          \\
\mbox{s.t.} & A_j^\top w_j = x, \\
                  & w_j \succeq 0,
\end{array}\]
we can use the strong duality property to reformulate \eqref{vproblem} as the following quadratic programming problem:
\begin{equation} \label{vproblem3}
 \begin{array}{cl}
\min_{\tau,u,w_j}   &  \tau+\frac{1}{2 \alpha}\|u - x^k\|^2      \\
\mbox{s.t.} & G_j(x^k)^\top(u-x^k) + b_j^\top w_j -H_{j}(x^k)  \leq \tau, \quad \forall j=1,\ldots,m,\\
                            & A_j^\top w_j = u, \quad \forall j=1,\ldots,m,\\
                            & w_j \succeq 0,\quad \forall j=1,\ldots,m, \\
                            & lb\preceq u\preceq ub.
\end{array}
\end{equation}
For more details, see \cite[Section~5.2 (a)]{TanabeFukudaYamashita2019}. In the codes, we use a simplex-dual method ({\it linprog} routine) to solve \eqref{evalh}, and an interior point method  ({\it quadprog} routine) to solve \eqref{vproblem3}.

In our experiments, we define the uncertainty set ${\cal{Z}}_{j}$ as follows: Let $B_j \in \RR^{n\times n}$ be a (random) nonsingular matrix, and $\delta > 0$ be a given parameter. Then
\begin{equation*}\label{Zdef}
{\cal{Z}}_{j} := \{ z \in \RR^n \mid  -\delta e \preceq B_jz\preceq \delta e\}, \quad \forall j=1,\ldots,m,
\end{equation*}
where  $e=(1,\ldots,1)^T\in\RR^n$.
The parameter $\delta$ plays a crucial role in controlling the uncertainty of the problem and is defined as:
\begin{equation*}\label{deltadef}
\delta := \hat{\delta} \|\hat x\|,
\end{equation*}
 where $\hat{\delta}$ is randomly chosen from the interval $[0.02,0.10]$, and $\hat x$ is an arbitrary point in $\dom(F)$.

 \subsection{Efficiency and robustness}

For each test problem, we employed 100 randomly generated starting points within the corresponding set $\dom(F)$. In this phase, each combination of a problem and a starting point was treated as an independent instance and solved by the three algorithms. A run is considered successful if an approximate weakly Pareto optimal point is found according to \eqref{eq:stop}, regardless of the objective function value. The results in Figure~\ref{fig:results} are presented using performance profiles~\cite{dolan2002benchmarking} and compare the algorithms with respect to: (a) CPU time; (b) number of evaluations of $H_j$’s functions. It is worth noting that, in a performance profile graph, the extreme left (at 1 in the domain) assesses {\it efficiency}, while the extreme right evaluates {\it robustness}.

\noindent\begin{figure}[H]
\centering \small
  \begin{tabular}{cc} 
  (a) CPU time &(b) $H_j$’s evaluations\\ 
\hspace{-12pt}\includegraphics[scale=\myscale]{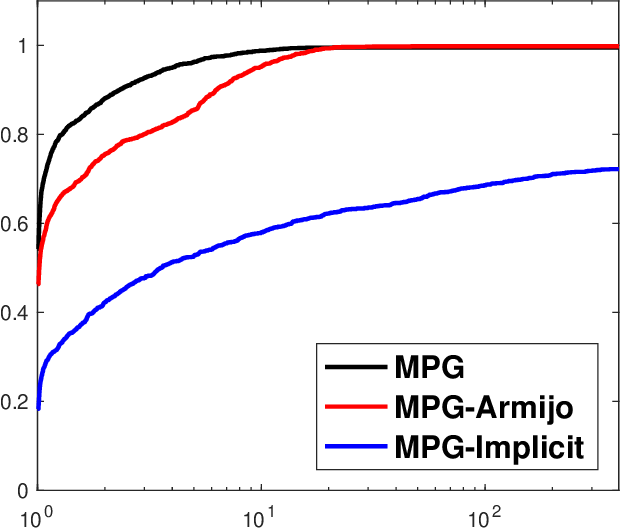}&\includegraphics[scale=\myscale]{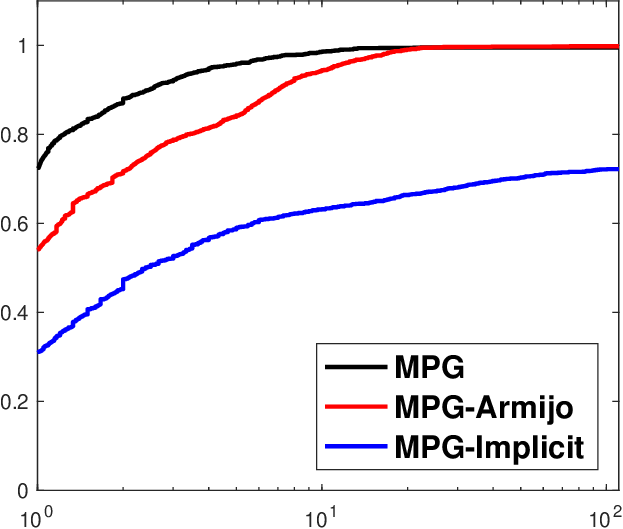} \\
\end{tabular}
\caption{Performance profiles considering: (a) CPU time; (b) number of evaluations of $H_j$’s functions.}
\label{fig:results}
\end{figure}

 As can be seen, MPG is the most efficient algorithm on the chosen set of test problems, followed by the MPG-Armijo algorithm and then the MPG-Implicit algorithm. The respective efficiencies of the methods are $54.3\%$ (resp. $72.5\%$), $46.4\%$ (resp. $54.2\%$), and $18.3\%$ (resp. $31.1\%$), considering CPU time (resp. number of $H_j$’s evaluations) as the performance measurement. The robustness values are $99.5\%$, $99.8\%$, and $72.2\%$, respectively; see Table~\ref{tab:perf}. Despite enjoying global convergence properties, the MPG-Implicit algorithm faced challenges in finding approximate weakly Pareto optimal points, often reaching the maximum allowed iterations. This limitation contrasts with the robust performance of the MPG and MPG-Armijo algorithms. Furthermore, the superiority of the MPG algorithm over the MPG-Armijo algorithm in terms of function evaluations highlights the potential advantage of using the new explicit line search procedure in problems in which evaluating the $H_j$'s functions is computationally expensive.

  \begin{table}[htb!] 
  \footnotesize
  \centering
\begin{tabular}{|c|c|c|} \hhline{~*2{-}|}  
\multicolumn{1}{c|}{} & \multicolumn{1}{c|}{\cellcolor[gray]{0.9} Efficiency (CPU time -- $H_j$’s evaluations) ($\%$) } & \multicolumn{1}{c|}{\cellcolor[gray]{0.9} Robustness ($\%$)} \\ \hline
\cellcolor[gray]{0.9} MPG            & 54.3  -- 72.5  & 99.5 \\ \hline
\cellcolor[gray]{0.9} MPG-Armijo     & 46.4 -- 54.2   & 99.8\\ \hline
\cellcolor[gray]{0.9} MPG-Implicit   & 18.3  -- 31.1    & 72.2\\ \hline
\end{tabular}
\caption{Efficiency and robustness of the algorithms on the chosen set of test problems.}
\label{tab:perf}
\end{table}

 \subsection{Pareto frontiers}

In multiobjective optimization, the main goal is to estimate the Pareto frontier of a given problem. A commonly employed strategy for this task involves running an algorithm from various starting points and collecting the Pareto optimal points found. Therefore, for a given test problem, we execute each algorithm for 2 minutes to obtain an approximation of the Pareto frontier. We compare the results using well-known metrics such as {\it Purity} and ($\Gamma$ and $\Delta$) {\it Spread}. In essence, for a given problem, the Purity metric measures the algorithm's ability to discover points on the Pareto frontier, while the Spread metric assesses its capability to obtain well-distributed points along the Pareto frontier. For a detailed discussion of these metrics and their uses along with performance profiles, see \cite{doi:10.1137/10079731X}. 
\noindent\begin{figure}[H]
\centering \small
 \begin{tabular}{ccc}
 (a) Purity &(b) Spread $\Gamma$&(c) Spread $\Delta$\\
\hspace{-12pt}\includegraphics[scale=\myscale]{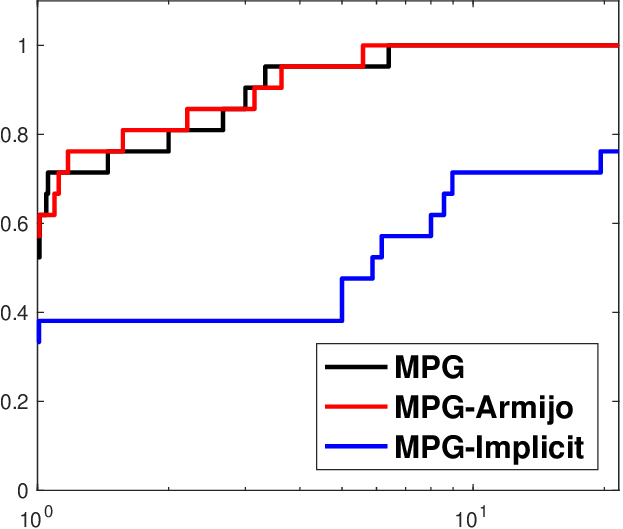}&\hspace{-12pt} \includegraphics[scale=\myscale]{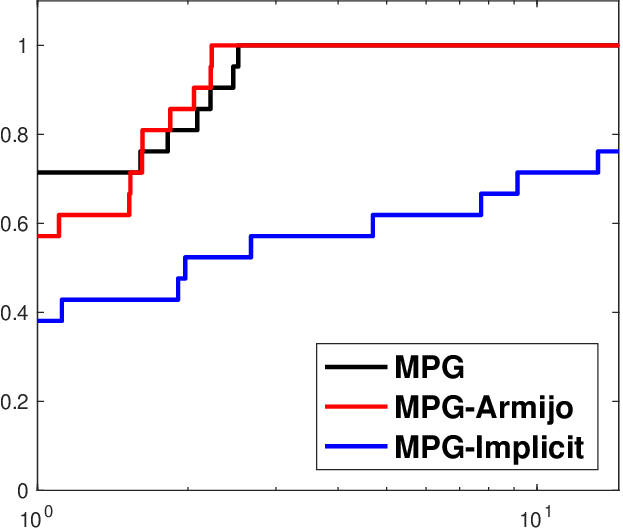}&\hspace{-12pt}\includegraphics[scale=\myscale]{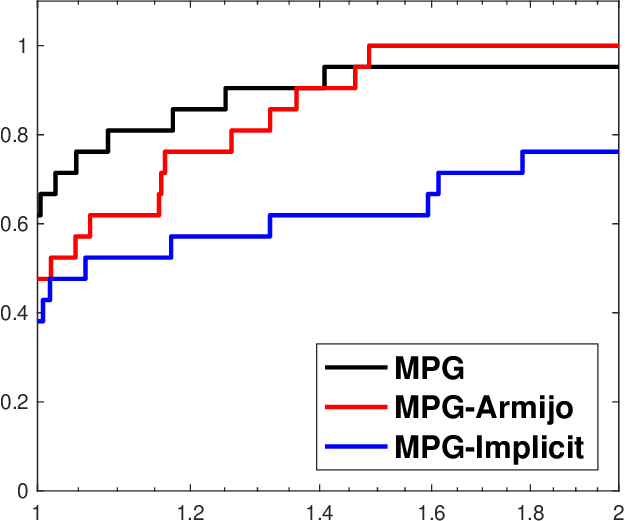}\\
\end{tabular}
\caption{Metric performance profiles: (a) Purity; (b) Spread $\Gamma$; (c) Spread $\Delta$.}
\label{fig:metrics}
\end{figure}
The results presented in above Figure~\ref{fig:metrics} clearly show that the MPG and MPG-Armijo algorithms outperform the MPG-Implicit algorithm. Regarding the Purity metric, both the MPG algorithm and the MPG-Armijo algorithm exhibited equivalent performance. However, concerning the Spread metrics, a slight advantage can be observed for the MPG algorithm. This allows us to conclude that using the new linesearch procedure does not compromise the practical performance of a proximal-gradient-type method, especially concerning solution quality. 

For illustrative purposes, Figure~\ref{fig:pareto} below displays the image space along with the approximation of the Pareto frontier obtained by the MPG algorithm for the BK1, Lov1, PNR, and VU2 problems. In the graphics, a solid dot denotes a final iterate, whereas the origin of a straight line segment indicates the corresponding starting point. As can be seen, the algorithm visually achieves a satisfactory outline of the Pareto frontiers.

\noindent\begin{figure}[h!]
\centering \small
  \begin{tabular}{cc} 
  \multicolumn{2}{c}{BK1}\\
\hspace{-12pt}\includegraphics[scale=\myscale]{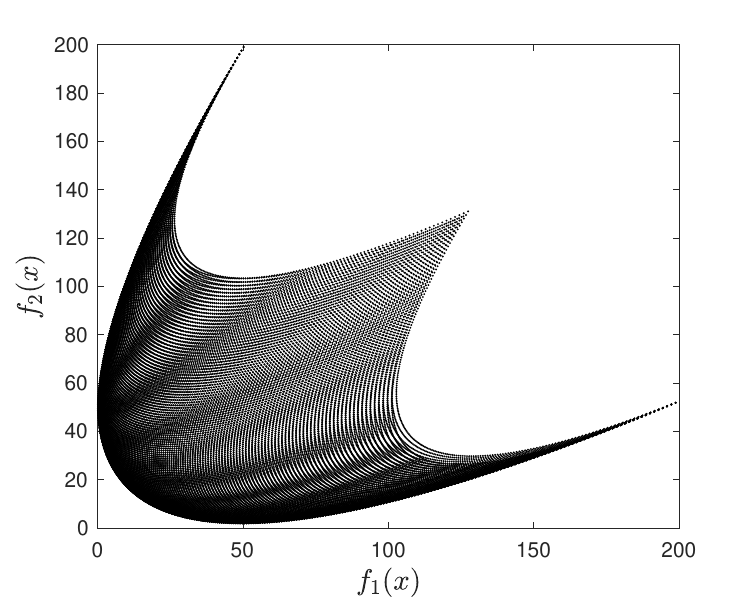}&\includegraphics[scale=\myscale]{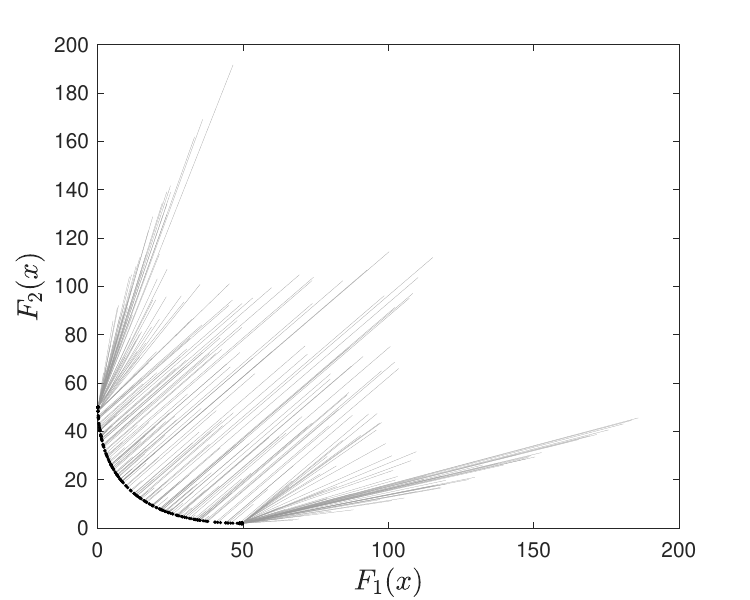} \\
  \multicolumn{2}{c}{Lov1}\\
\hspace{-12pt}\includegraphics[scale=\myscale]{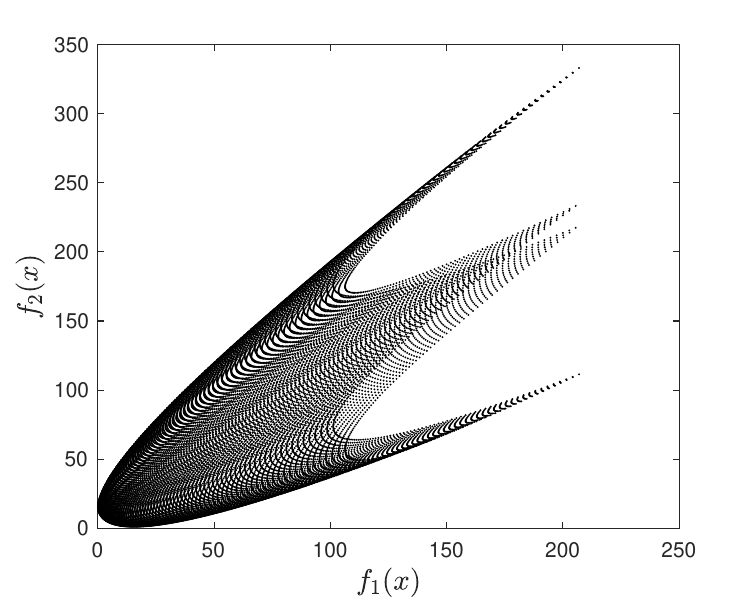}&\includegraphics[scale=\myscale]{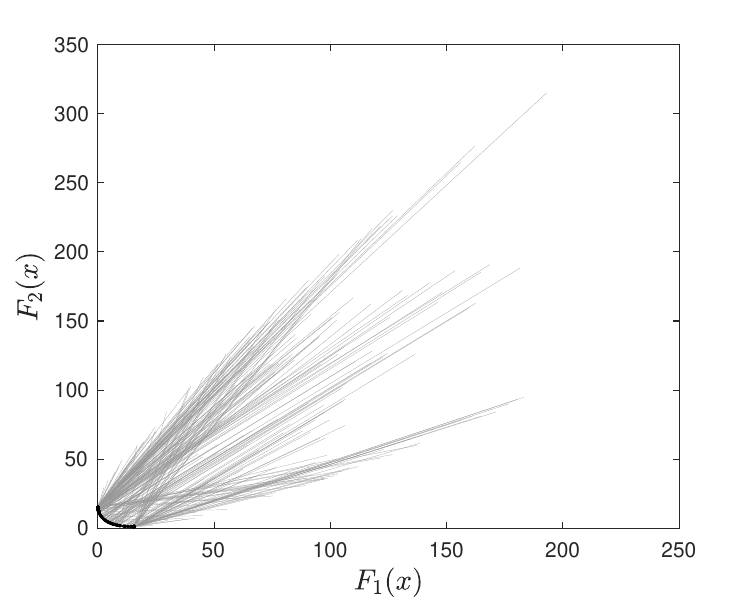} \\
  \multicolumn{2}{c}{PNR}\\
 \hspace{-12pt}\includegraphics[scale=\myscale]{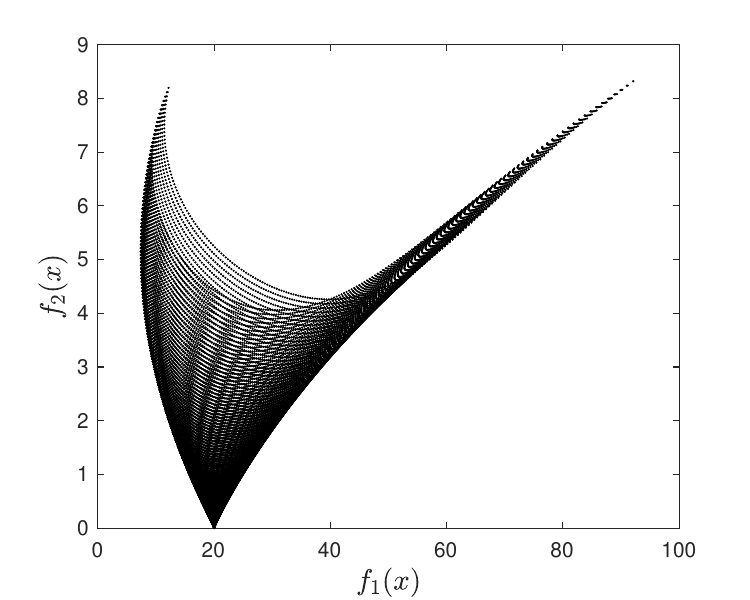}&\includegraphics[scale=\myscale]{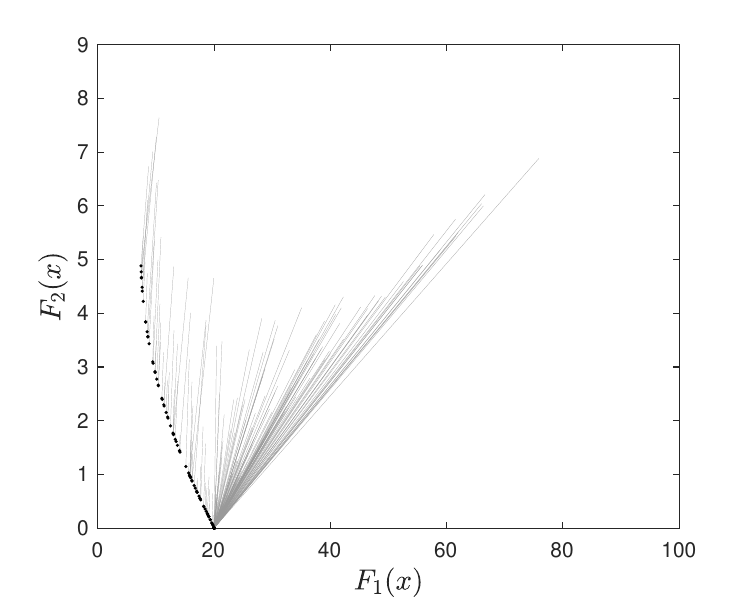} \\
   \multicolumn{2}{c}{VU2}\\
\hspace{-12pt}\includegraphics[scale=\myscale]{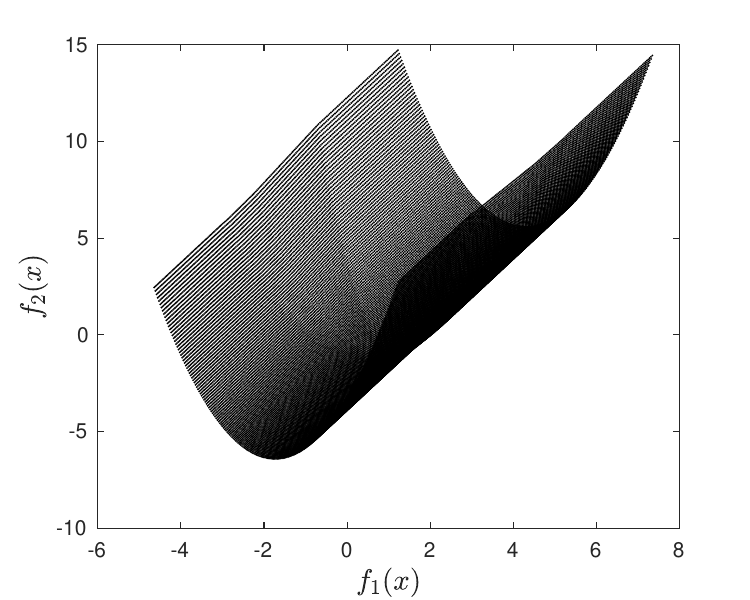}&\includegraphics[scale=\myscale]{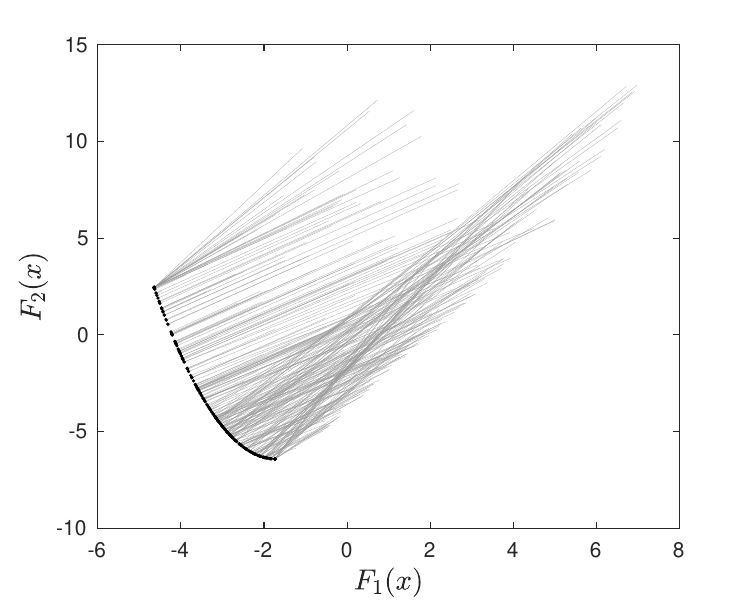} \\
\end{tabular}
\caption{Image spaces and approximations of the Pareto frontiers obtained by the MPG algorithm for the BK1, Lov1, PNR, and VU2 problems.}
\label{fig:pareto}
\end{figure}


\section{Conclusion}\label{sec:concluding}

We presented a proximal gradient method with a new explicit line search procedure designed for convex multi-objective optimization problems, where each objective function $F_j$ is of the form $F_j := G_j + H_j$, with $G_j$ and $H_j$ being convex functions and $G_j$ being continuously differentiable. The algorithm requires solving only one proximal subproblem per iteration, with the backtracking scheme exclusively applied to the differentiable functions $G_j$. Our numerical experiments revealed that the proposed method exhibits a reduced number of evaluations of $H_j$ when compared to the proximal gradient algorithm with Armijo line search \cite{TanabeFukudaYamashita2019}. This characteristic makes our approach particularly promising in applications where the computation of $H_j$ is computationally expensive. Future investigations could explore the combination of the proposed line search procedure with other methods, such as Newtonian algorithms. 


\end{document}